\documentclass{amsart}          

\usepackage{amsmath,amsthm}     
\usepackage{amssymb}            
\usepackage{euscript}           
\usepackage{enumerate,calc}
\usepackage[matrix,arrow,curve,frame]{xy}    

\xymatrixcolsep{1.9pc}                          
\xymatrixrowsep{1.9pc}
\newdir{ >}{{}*!/-5pt/\dir{>}}                  

\numberwithin{equation}{section}

\newtheorem{thm}[subsection]{Theorem}

\newtheorem{prop}[subsection]{Proposition}

\newtheorem{cor}[subsection]{Corollary}
\newtheorem{lemma}[subsection]{Lemma}

\theoremstyle{definition}  

\newtheorem{remark}[subsection]{Remark}

\newtheorem{rmk}[subsection]{Remark}
\newtheorem*{lemma*}{Lemma}
\newtheorem*{thm*}{Theorem}

  {\end{list}}

\newcommand{\chapter}{\section}

\newcommand{\iso}               {\cong}

\newcommand{\cat}{\EuScript}    
\newcommand{\cA}{{\cat A}}      

\newcommand{\cE}{{\cat E}}

\newcommand{\cH}{{\cat H}}

\newcommand{\cO}{{\mathcal O}}

\newcommand{\cU}{{\cat U}}
\newcommand{\Top}{{\cat Top}}

\newcommand{\sSet}{s{\cat Set}}

\newcommand{\Ab}{{\cat Ab}}

\newcommand{\field}[1]  {\mathbb #1} 

\newcommand{\R}         {\field R}

\newcommand{\Z}         {\field Z}

\newcommand{\bP}        {\field P}

\DeclareMathOperator*{\colim}{colim}
\DeclareMathOperator*{\hocolim}{hocolim}

\newcommand{\ra}{\rightarrow}                   
\newcommand{\lra}{\longrightarrow}              
\newcommand{\la}{\leftarrow}                    
\newcommand{\llra}[1]{\stackrel{#1}{\lra}}      

\newcommand{\we}{\llra{\sim}}                   

\newcommand{\tuborg}{\left\{\begin{array}{ll}}
\newcommand{\sluttuborg}{\end{array}\right.}

\begin{document}

\title{The $RO(G)$-Graded Serre Spectral Sequence}

\author{William Kronholm}
\address{Department of Mathematics and Statistics\\ Swarthmore College\\ Swarthmore, PA 19081 }

\date{\today}
\begin{abstract}
In this paper the Serre spectral sequence of \cite{MS} is extended
from Bredon cohomology to $RO(G)$-graded cohomology for finite
groups $G$.  Special attention is paid to the case $G=\Z/2$ where
the spectral sequence is used to compute the cohomology of certain
projective bundles and loop spaces.


\end{abstract}
\maketitle

\tableofcontents


\section{Introduction}
\label{intro}

In \cite{Bredon}, Bredon created equivariant homology and cohomology
theories of $G$-spaces, now called Bredon homology and Bredon
cohomology, which yield the usual singular homology and cohomology
theories when the group acting is taken to be the trivial group.  In
\cite{LMM}, a cohomology theory for $G$-spaces is constructed that
is graded on $RO(G)$, the Grothendieck ring of virtual
representation of $G$.  This $RO(G)$-graded theory extends Bredon
cohomology in the sense that $H^{\underline{n}}(X)=H^n_{Br}(X)$ when
$\underline{n}$ is the trivial $n$-dimensional representation of
$G$.

Many of the usual tools for computing cohomolgy have their
counterparts in the $RO(G)$-graded setting.  These include
Mayer-Vietoris sequences, K\"unneth theorem, suspension
isomorphisms, etc.  Missing from the $RO(G)$ computational tool box
was an equivariant version of the Serre spectral sequence associated
to a fibration $F \ra E \ra B$.  Also, perhaps partially because of
a lack of this spectral sequence, the theory of equivariant
characteristic classes has not yet been developed.

The main result of this paper is to extend the spectral sequence of
a $G$-fibration given in \cite{MS} from Bredon cohomology to the
$RO(G)$-graded theory with special attention to the case $G=\Z/2$.  The restriction to $G=\Z/2$ is for two reasons.  The first is that there is a map from Voevodsky's motivic cohomology and $RO(\Z/2)$-graded equivariant cohomology, and so one can try to answer questions in motivic cohomology by considering instead the relevant equivariant cohomology.  The second reason is that the general algebra of Mackey functors is extremely complicated for arbitrary groups (even compact Lie), yet the $G=\Z/2$ case is manageable.

A $p$-dimensional real $\Z/2$-representation $V$ decomposes as
$V=(\R^{1,0})^{p-q} \oplus (\R^{1,1})^q =\R^{p,q}$ where $\R^{1,0}$
is the trivial representation and $\R^{1,1}$ is the nontrivial
1-dimensional representation.  Thus the $RO(\Z/2)$-graded theory is
a bigraded theory, one grading measuring dimension and the other
measuring the number of ``twists''.  In this case, we write
$H^{V}(X;M)=H^{p,q}(X;M)$ for the $V^{\text{th}}$ graded component
of the $RO(\Z/2)$-graded equivariant cohomology of $X$ with
coefficients in a Mackey functor $M$. Here is the spectral sequence:
\begin{thm*}
If $f\colon E \ra X$ is a fibration of $\Z/2$ spaces, then for every
$r \in \Z$ and every Mackey functor $M$ there is a natural spectral
sequence with $$E^{p,q}_2=H^{p,0}(X;\cH^{q,r}(f;M)) \Rightarrow
H^{p+q,r}(E;M).$$
\end{thm*}

This is a spectral sequence that takes as inputs the Bredon
cohomology of the base space with coefficients in the local
coefficient system $\cH^{q,r}(f;M)$ and converges to the
$RO(\Z/2)$-graded cohomology of the total space.

This is really a family of spectral sequences, one for each integer
$r$. If the Mackey functor $M$ is a ring Mackey functor, then this
family of spectral sequences is equipped with a tri-graded
multiplication. If $a \in  H^{p,0}(X;\cH^{q,r}(f;M))$ and $b \in
H^{p',0}(X;\cH^{q',r'}(f;M))$, then $a\cdot b \in
H^{p+p',0}(X;\cH^{q+q',r+r'}(f;M))$.  There is also an action of
$H^{*,*}(pt;M)$ so that if $\alpha \in H^{q',r'}(pt;M)$ and $a \in
H^{p,0}(X;\cH^{q,r}(f;M))$, then $\alpha \cdot a \in
H^{p,0}(X;\cH^{q+q',r+r'}(f;M))$.

Under certain connectivity assumptions on the base space, the local
coefficients $\cH^{q,r}(f;M)$ are constant, and the spectral
sequence becomes the following. This result is restated and proved
as Theorem \ref{thm:ss}.
\begin{thm*}
If $X$ is equivariantly 1-connected and $f \colon E \ra X$ is a
fibration of $\Z/2$ spaces with fiber $F$, then for every $r \in \Z$
and every Mackey functor $M$ there is a spectral sequence with
$$E^{p,q}_2=H^{p,0}(X;\underline{H}^{q,r}(F;M)) \Rightarrow
H^{p+q,r}(E;M).$$
\end{thm*}

The coefficient systems $\cH^{q,r}(f;M)$ and
$\underline{H}^{q,r}(F;M)$ that appear in the spectral sequence are
explicitly defined in the next section.  They are the equivariant
versions of the usual local coefficient systems that arise when
working with the usual Serre spectral sequence.

This spectral sequence is rich with information about the fibration
involved, even in the case of the trivial fibration $id \colon X \ra
X$.  In this case, the $E_2$ page takes the form
$E^{p,q}_2=H^{p,0}(X;\underline{H}^{q,r}(pt;M)) \Rightarrow
H^{p+q,r}(X;M)$.  Set $M = \underline{\Z/2}$ and consider the case
$r=1$. Then $H^{p,0}(X;\underline{H}^{q,r}(pt;\underline{\Z/2}))=0$
if $q \neq 0,1$.  The case $q=0$ gives
$H^{p,0}(X;\underline{H}^{0,1}(pt;\Z/2))=H^{p,0}(X;\underline{\Z/2})$,
and if $q=1$,
$H^{p,0}(X;\underline{H}^{1,1}(pt;\underline{\Z/2}))=H^p_{sing}(X^G;\Z/2)$.
The spectral sequence then has just two non-zero rows as shown in
Figure \ref{fig:ssex} below.

\begin{figure}[htpb]
\centering
\begin{picture}(340,180)(-30,-10)
\multiput(0,0)(70,0){5}{\line(0,1){150}}
\multiput(0,0)(0,30){6}{\line(1,0){280}}
\put(8,10){$H^{0,0}(X)$} \put(8,40){$H^{0}_{sing}(X^G)$}
\put(30,70){$0$} \put(30,100){$0$} \put(30,130){$0$}

\put(78,10){$H^{1,0}(X)$} \put(78,40){$H^{1}_{sing}(X^G)$}
\put(99,70){$0$} \put(99,100){$0$} \put(99,130){$0$}

\put(148,10){$H^{2,0}(X)$} \put(148,40){$H^{2}_{sing}(X^G)$}
\put(169,70){$0$} \put(169,100){$0$} \put(169,130){$0$}

\put(218,10){$\cdots$} \put(218,40){$\cdots$}

\put(-2,-2){$\bullet$} \thicklines \put(0,0){\vector(0,1){155}}
\put(0,0){\vector(1,0){285}}

\put(280,-10){$p$} \put(-10,150){$q$} \thinlines
\end{picture}
\caption{The $r=1$ spectral sequence for $id \colon X \ra X$.}
\label{fig:ssex}
\end{figure}

As usual, the two row spectral sequence yields the following curious
long exact sequence:

$0 \ra H^{0,0}(X) \ra H^{0,1}(X) \ra 0 \ra H^{1,0}(X) \ra H^{1,1}(X)
\ra H^0_{sing}(X^G) \ra H^{2,0}(X) \ra H^{2,1}(X) \ra
H^1_{sing}(X^G) \ra \cdots $

Now, to any equivariant vector bundle $f\colon E \ra X$, there is an
associated equivariant projective bundle $\bP(f)\colon \bP(E)\ra X$
whose fibers are lines in the fibers of the original bundle.
Applying the above spectral sequence to this new bundle yields the
following result, which appears later as Theorem \ref{thm:localss}.
\begin{thm*}
If $X$ is equivariantly 1-connected and $f\colon E \ra X$ is a
vector bundle with fiber $\R^{n,m}$ over the base point, then the
spectral sequence of Theorem \ref{thm:ss} for the bundle
$\bP(f)\colon \bP(E)\ra X$ with constant $M=\underline{\Z/2}$
coefficients ``collapses".
\end{thm*}

Here, when we say the spectral sequence collapses, we do not mean it
collapses in the usual sense.  Each fibration $f\colon E \ra X$ maps
to the trivial fibration $id \colon X \ra X$ in an obvious way.
Naturality then provides a map from the spectral sequence for $id
\colon X \ra X$ to the spectral sequence for $f\colon E \ra X$.  In
the above theorem, the spectral sequence ``collapses" in the sense
that the only nonzero differentials are those arising from the
trivial fibration $id \colon X \ra X$.

In non-equivariant topology, the Leray-Serre spectral sequence gives
rise to a description of characteristic classes of vector bundles.
Consider the universal bundle $E_n \ra G_n$ over the Grassmannian of
$n$-planes in $\R^\infty$.  Forming the associated projective bundle
$\bP(E_n) \ra G_n$ yields a fiber bundle with fiber $\R\bP^\infty$.
Applying the Leray-Serre spectral sequence to this projective bundle
yields characteristic classes of $E_n$ as the image of the
cohomology classes $1, z, z^2, \dots \in
H^*_{sing}(\R\bP^\infty;\Z/2)$ under the transgressive
differentials.  Since this universal bundle classifies vector
bundles, characteristic classes of arbitrary bundles can be defined
as pullbacks of the characteristic classes, $c_i \in H^i(G_n;\Z/2)$,
of the universal bundle.  It would be nice to adapt this
construction to the $\Z/2$ equivariant setting.  However, the
equivariant space $G_n((\R^{2,1})^\infty) = G_n(\cU) = G_n$ is not
1-connected, and so the spectral sequence is not as easy to work
with.  It seems that there is no way to avoid using local
coefficient systems in this setting.

Section \ref{sec:prelims} provides some of the definitions and
basics that are needed for this paper.  The main theorem is stated
and proved in section \ref{sec:SS}, making use of some technical
homotopical details that are provided in section
\ref{sec:homotopical}.  In section \ref{sec:ProjBundles}, the
spectral sequence is then applied to compute the cohomology of a
projective bundle $\bP(E)$ associated to a vector bundle $E \ra X$.
Section \ref{sec:AH} provides an application of the
$RO(\Z/2)$-graded Serre spectral sequence to loop spaces on certain
spheres.

\section{Preliminaries}
\label{sec:prelims}

\label{chapter:prelims} The section contains some of the basic
machinery and notation that will be used throughout the paper.  In
this section, let $G$ be any finite group.

A $G$-CW complex is a $G$-space $X$ with a filtration $X^{(n)}$
where $X^{(0)}$ is a disjoint union of $G$-orbits and $X^{(n)}$ is
obtained from $X^{(n-1)}$ by attaching cells of the form $G/H_\alpha
\times \Delta^n$ along maps $f_\alpha \colon G/H_\alpha \times
\partial\Delta^n \ra X^{(n-1)}$.  The space $X^{(n)}$ is referred to
as the $n$-skeleton of $X$.  Such a filtration on a space $X$ is
called a cell structure for $X$.

Given a $G$-representation $V$, let $D(V)$ and $S(V)$ denote the
unit disk and unit sphere, respectively, in $V$ with action induced
by that on $V$.  A $\text{Rep}(G)$-complex is a $G$-space $X$ with a
filtration $X^{(n)}$ where $X^{(0)}$ is a disjoint union of
$G$-orbits and $X^{(n)}$ is obtained from $X^{(n-1)}$ by attaching
cells of the form $D(V_\alpha)$ along maps $f_\alpha \colon
S(V_\alpha) \ra X^{(n-1)}$ where $V_\alpha$ is an $n$-dimensional
real representation of $G$.  The space $X^{(n)}$ is again referred
to as the $n$-skeleton of $X$, and the filtration is referred to as
a cell structure.

Let $\Delta_G(X)$ be the category of equivariant simplices of the
$G$-space $X$.  Explicitly, the objects of $\Delta_G(X)$ are maps
$\sigma\colon G/H \times \Delta^n \ra X$.  A morphism from $\sigma$
to $\tau\colon G/K \times \Delta^m \ra X$ is a pair $(\varphi,
\alpha)$ where $\varphi\colon G/H \ra G/K$ is a $G$-map and
$\alpha\colon \Delta^n \ra \Delta^m$ is a simplicial operator such
that $\sigma = \tau \circ (\varphi \times \alpha)$.

Let $\Pi_G(X)$ be the fundamental groupoid of $X$.  Explicitly, the
objects of $\Pi_G(X)$ are maps $\sigma\colon G/H \ra X$ and a
morphism from $\sigma$ to $\tau\colon  G/K \ra X$ is a pair
$(\varphi, \alpha)$ where $\varphi \colon G/H \ra G/K$ is a $G$-map
and $\alpha$ is a $G-$homotopy class of paths from $\sigma$ to $\tau
\circ \varphi$.

There is a forgetful functor $\pi\colon  \Delta_G(X) \ra \Pi_G(X)$
that sends $\sigma\colon G/H \times \Delta^n \ra X$ to $\sigma\colon
G/H \ra X$ by restricting to the last vertex $e^n$ of $\Delta^n$.  A
morphism $(\varphi, \alpha)$ in $\Delta_G(X)$ is restricted to
$(\varphi, \alpha)$ in $\Pi_G(X)$ by restricting $\alpha$ to the
linear path from $\alpha(e^n)$ to $e^m$ in $\Delta^m$. There is a
further forgetful functor to the orbit category $\cO(G)$, which will
also be denoted by $\pi$, as shown below.

\[\xymatrix{\Delta_G(X) \ar[r]^{\pi} & \Pi_G(X) \ar[r]^{\pi} & \cO(G)}\]

A coefficient system on $X$ is a functor $M\colon \Delta_G(X)^{op}
\ra \Ab$.  We say that the coefficient system $M$ is a local
coefficient system if it factors through the forgetful functor to
$\Pi_G(X)^{op}$ (up to isomorphism).  If $M$ further factors through
$\cO(G)^{op}$, then we call $M$ a constant coefficient system.

For the precise definition of a Mackey functor for $G=\Z/2$, the
reader is referred to \cite{Alaska} or \cite{DuggerKR}.  A summary
of the important aspects of a $\Z/2$ Mackey functor is given here.
The data of a Mackey functor are encoded in a diagram like the one
below.

\[\xymatrix{ M(\Z/2) \ar@(ur,ul)[]^{t^*} \ar@/^/[r]^(0.6){i_*} & M(e) \ar@/^/[l]^(0.4){i^*}} \]

The maps must satisfy the following four conditions.
\begin{enumerate}
\item $(t^*)^2 = id$
\item $t^*i^*=i^*$
\item $i_*(t^*)^{-1}=i_*$
\item $i^*i_*=id+t^*$
\end{enumerate}

According to \cite{LMM}, each Mackey functor $M$ uniquely
determines an $RO(G)$-graded cohomology theory characterized by
\begin{enumerate}
\item $H^{\underline{n}}(G/H;M) =\begin{cases}
M(G/H) & \text{ if } n=0 \\
0 & \text{otherwise}\end{cases}$
\item The map $H^0(G/K;M) \ra H^0(G/H;M)$ induced by $i \colon G/H \ra G/K$ is the transfer map $i^*$ in the Mackey functor.
\end{enumerate}

Given a Mackey functor $M$, a $G$-representation $V$, and a
$G$-space $X$, we can form a coefficient system
$\underline{H}^V(X;M)$.  This coefficient system is determined on
objects by $\underline{H}^V(X;M)(G/H)=H^V(X \times G/H;M)$ with maps
induced by those in $\cO(G)$.

In this paper, $G$ will usually be $\Z/2$ and the Mackey functor
will almost always be constant $M=\underline{\Z/2}$ which has the
following diagram.

\[\xymatrix{ \Z/2 \ar@(ur,ul)[]^{id} \ar@/^/[r]^{0} & \Z/2 \ar@/^/[l]^{id}} \]

With these constant coefficients, the $RO(\Z/2)$-graded cohomology
of a point is given by the picture in Figure \ref{fig:pt}.

\begin{figure}[htpb]
\centering
\begin{picture}(100,100)(-100,-100)
\put(-100,-50){\vector(1,0){100}}
\put(-50,-100){\vector(0,1){100}}

\put(-50, -51){\line(0,1){40}} \put(-50, -51){\line(1,1){40}}
\put(-50, -91){\line(0,-1){20}} \put(-50, -91){\line(-1,-1){20}}
\put(-50, -51){\circle*{2}}

\multiput(-90,-51)(20,0){5}{\line(0,1){3}}
\multiput(-52,-90)(0,20){5}{\line(1,0){3}}
\put(-49,-57){$\scriptscriptstyle{0}$}
\put(-29,-57){$\scriptscriptstyle{1}$}
\put(-9,-57){$\scriptscriptstyle{2}$}
\put(-69,-57){$\scriptscriptstyle{-1}$}
\put(-89,-57){$\scriptscriptstyle{-2}$}
\put(-57,-47){$\scriptscriptstyle{0}$}
\put(-57,-27){$\scriptscriptstyle{1}$}
\put(-57,-7){$\scriptscriptstyle{2}$}

\put(-48,-35){$\tau$} \put(-30,-35){$\rho$} \put(-10,-15){$\rho^2$}
\put(-30,-17){$\tau\rho$} \put(-58, -90){$\theta$} \put(-48,
-102){$\frac{\theta}{\tau}$} \put(-78, -102){$\frac{\theta}{\rho}$}

\put(-58,0){${q}$} \put(0,-60){${p}$}

\multiput(-50, -30)(20,20){2}{\circle*{2}} \multiput(-50,
-50)(20,20){3}{\circle*{2}} \put(-50, -10){\circle*{2}} \put(-50,
-90){\circle*{2}} \put(-50, -110){\circle*{2}} \put(-70,
-110){\circle*{2}}

\end{picture}

\caption{$H^{*,*}(pt;\underline{\Z/2})$} \label{fig:pt}
\end{figure}
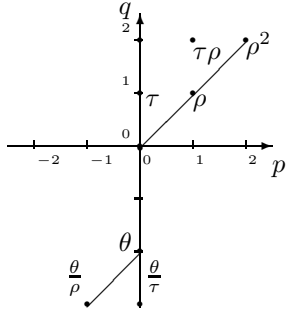

Every lattice point in the picture that is inside the indicated
cones represents a copy of the group $\Z/2$.  The top cone is a
polynomial algebra on the elements $\rho \in
H^{1,1}(pt;\underline{\Z/2})$ and $\tau \in
H^{0,1}(pt;\underline{\Z/2})$.  The element $\theta$ in the bottom
cone is infinitely divisible by both $\rho$ and $\tau$.  The
cohomology of $\Z/2$ is easier to describe:
$H^{*,*}(\Z/2;\underline{\Z/2})=\Z/2[t, t^{-1}]$ where $t \in
H^{0,1}(\Z/2;\underline{\Z/2})$.  Details can be found in
\cite{DuggerKR} and \cite{Caruso}.

Given a $G$-map $f\colon E \ra X$ and a Mackey functor $M$, we can
define a coefficient system $\cH^{q,r}(-,M)\colon  \Delta_G(X)^{op}
\ra \Ab$ by taking cohomology of pullbacks:
$\cH^{q,r}(f,M)(\sigma)=H^{q,r}(\sigma^*(E), M)$.  In \cite{MS}, it
is shown that this is a local coefficient system when $f$ is a
$G$-fibration.

Given a $G$-fibration $f\colon E \ra X$, we can define a functor
$\Gamma_f\colon \Delta_G(X) \ra \Top$.  On objects,
$\Gamma(\sigma)=\sigma^*(E)$. On morphisms, $\Gamma(\varphi,
\alpha)=\overline{\varphi \times \alpha}$, where $\overline{\varphi
\times \alpha}$ is the map of total spaces in the diagram

\[ \xymatrix{ \sigma^*(E)\ar[d] \ar[r]^{\overline{\varphi \times \alpha}} & \tau^*(E)\ar[d] \\
                G/H \times \Delta^n  \ar[r]^{\varphi \times \alpha} & G/K \times \Delta^m.}
\]

Let $f\colon E \ra X$ be a $G$-fibration over an equivariantly
1-connected $G$-space $X$ with base point $x \in X$ and let
$F=f^{-1}(x)$.  Define a constant coefficient system
$\underline{H}^{q,r}(F;M)$ as follows:
$\underline{H}^{q,r}(F;M)(G/H) = H^{q,r}((G/H) \times F;M)$ and if
$\varphi\colon  G/H \ra G/K$ is a $G$-map, then
$\underline{H}^{q,r}(F;M)(\varphi) = (\varphi \times id)^*$.  It is
this coefficient system that appears in the spectral sequence of
Theorem \ref{thm:localss}.

\section{Construction of the Spectral Sequence}
\label{sec:SS}

Unlike in ordinary topology, the equivariant Serre spectral sequence
for a fibration $f \colon E \ra X$ will not be deduced from lifting
a cellular filtration of $X$ to one on $E$.  Instead, the spectral
sequence is a special case of the one for a homotopy colimit. Recall
(from \cite{hocolim} for example) that given a cohomology theory
$\cE^*$ and a diagram of spaces $D \colon I \ra \Top_G$, there is a
natural spectral sequence as follows:
\begin{equation}
\label{eq:hocolimss} E_2^{p,q} = H^p(I^{op};\cE^q(D)) \implies
\cE^{p+q}(\hocolim D).
\end{equation}
For the case $I=\Delta_G(X)$, we know from \cite{MS} that the
cohomology of $\Delta_G(X)^{op}$ is the same as Bredon cohomology.
For a $G$-fibration $f \colon E \ra B$, we can consider the diagram
$\Gamma_f \colon \Delta_G(X) \ra \Top_G$ that sends $\sigma \colon
G/H \times \Delta^n \ra X$ to the pullback $\sigma^*(E)$.  We then
have the following technical lemma, whose proof is given in the next
section where it appears as Lemma \ref{lemma:hocowe}.
\begin{lemma*}
The composite $\hocolim_{\Delta_G(X)} \Gamma_f \ra
\colim_{\Delta_G(X)} \Gamma_f \ra E$ is a weak equivalence.
\end{lemma*}
Here is the desired spectral sequence.
\begin{thm}
\label{thm:ss} If $f\colon E \ra X$ is a fibration of $G$ spaces,
then for every $V \in RO(G)$ and every Mackey Functor $M$ there is a
natural spectral sequence with
$E^{p,q}_2(M,V)=H^{p,0}(X;\cH^{V+q}(f;M)) \Rightarrow
H^{V+p+q}(E;M)$.
\end{thm}

\begin{proof}
The homotopy colimit spectral sequence of (\ref{eq:hocolimss})
associated to $\Gamma_f$ and the cohomology theory $H^{V+*}(-;M)$
takes the form

\[E_2^{p,q}(M,V)=H^p(\Delta_G(X); \cH^{V+q}(\Gamma_f;M)) \Rightarrow H^{V+p+q}(\hocolim(\Gamma_f);M).\]

By \cite[Theorem 3.2]{MS} and Lemma \ref{lemma:hocowe}, this
spectral sequence becomes

\[E_2^{p,q}(M,V)=H^{p,0}(X; \cH^{V+q}(f;M)) \Rightarrow H^{V+p+q}(E;M).\]

Naturality of this spectral sequence follows from the naturality of
the homotopy colimit spectral sequence.

\end{proof}

The standard multiplicative structure on the spectral sequence is
given by the following theorem.  Recall that the analogue of tensor
product for Mackey functors is the box product, denoted by $\Box$.
See, for example, \cite{FL} for a full description of the box
product.

\begin{thm}
Given a $G$-fibration $f \colon E \ra X$, Mackey functors $M$ and
$M'$ and $V,V' \in RO(G)$, there is a natural pairing of the
spectral sequences of \ref{thm:ss}

$$E_r^{p,q}(M,V) \otimes E_r^{p',q'}(M';V') \ra E_r^{p+p',q+q'}(M\Box M'; V+V')$$

\noindent converging to the standard pairing

$$\cup \colon H^*(E;M) \otimes H^*(E;M') \ra H^*(E;M\Box M').$$

\noindent Furthermore, the pairing of $E_2$ terms agrees, up to a
sign $(-1)^{p'q}$, with the standard pairing
\begin{center}
$\xymatrix{H^{p,0}(X;\cH^{V+q}(f;M)) \otimes
H^{p',0}(X;\cH^{V'+q'}(f;M')) \ar[d]_\cup\\
H^{p+p',0}(X;\cH^{V+V'+p+p'+q+q'}(f;M\Box M'))}$
\end{center}
\end{thm}

\begin{proof}
This is a straightforward application of \cite[Theorem 4.1]{MS}.

\end{proof}

\remark If $M$ is a ring Mackey functor, then the product $M \Box M
\ra M$ gives a pairing of spectral sequences
$$E_r^{p,q}(M,V) \otimes E_r^{p',q'}(M,V') \ra E_r^{p+p',q+q'}(M, V+V').$$

\remark Since every $G$-fibration $f \colon E \ra X$ maps to the
$G$-fibration $id\colon X \ra X$, every spectral sequence of Theorem
\ref{thm:ss} admits a map from the spectral sequence for the
identity of $X$.

\begin{lemma}
\label{lemma:local} If $f\colon E \ra X$ is a $G$-fibration over an
equivariantly 1-connected based $G$-space X, then any local
coefficient system $\cA$ on $X$ is constant.
\end{lemma}
\begin{proof}

Choose a base point $x \in X$.  Then $x$ can be considered as a map
$x\colon  G/G \ra X$.  Denote by $x_H$ the point $x$ thought of as a
$G/H$ point.  That is $x_H=x \circ \pi$ where $\pi\colon G/H \ra
G/G$ is the projection. Notice that if $\varphi\colon  G/H \ra G/K$,
then $x_K = x_H \circ \varphi$.

Define a constant coefficient system $\bar{\cA} \colon \cO(G) \ra
\Ab$ by $\bar{\cA}(G/H)=\cA(x_H)$ and $\bar{\cA}(\varphi \colon  G/H
\ra G/K)= \cA(\varphi, c_x)$, where $c_x$ is the constant path from
$x_H$ to $x_K$.  The claim is that $\cA$ factors through $\bar{\cA}$
up to isomorphism.

For any object $\sigma\colon  G/H \ra X$, the connectivity
assumptions ensure that there is one homotopy class of paths from
$\sigma$ to $x_H$.  Let $\beta_\sigma$ be a representative path.

For any morphism $(\varphi, \alpha)$ in $\Pi_G(X)$ from $\sigma$ to
$\tau$, one then has the following commutative diagram:

\[ \xymatrix{\cA(\tau)  \ar[r]^{\cA(\varphi,\alpha)} & \cA(\sigma) \\
                  \cA(x_K) \ar[u]^{\cA(id,\beta_\tau)} \ar[r]^{\cA(\varphi,c_x)} & \cA(x_H) \ar[u]_{\cA(id, \beta_\sigma)} .}
\]

Now, each of the vertical maps is an isomorphism since, for example,
the path $\beta_\sigma$ has the inverse path
$\overline{\beta_\sigma}$ and each of the compositions $\beta_\sigma
* \overline{\beta_\sigma}$ and $\overline{\beta_\sigma} *
\beta_\sigma$ are homotopic to constant paths.  The same is true for
$\tau$.

Moreover, $(\bar{\cA} \circ \pi)(\sigma) = \bar{\cA}(G/H) =
\cA(x_H)$, and $(\bar{\cA} \circ \pi)(\varphi,
\alpha)=\bar{\cA}(\varphi) = \cA(\varphi, c_x)$.

This means that the above diagram exhibits an isomorphism from
$\bar{\cA} \circ \pi \ra \cA$.

\end{proof}

\begin{thm}
\label{thm:localss} If $X$ is equivariantly 1-connected and $f
\colon E \ra X$ is a fibration of $G$ spaces with fiber $F$, then
for every $V \in RO(G)$ and every Mackey Functor $M$ there is a
spectral sequence with
$E^{p,q}_2=H^{p,0}(X;\underline{H}^{V+q}(F;M)) \Rightarrow
H^{V+p+q}(E;M)$.
\end{thm}
\begin{proof} By Theorem \ref{thm:ss} and the above Lemma \ref{lemma:local}, it suffices to show that for the local coefficient system $\cA = \cH^{V+q}(f;M)$, the associated constant coefficient $\bar{\cA}$ is $\underline{H}^{V+q}(F;M)$.

Notice that $x_H = x \circ \pi$ and so $x_H^*E = (x \circ \pi)^*E =
\pi^* x^*E= \pi^*F = (G/H) \times F$.

We then have $\bar{\cA}(G/H) = \cA(x_H)=H^{V+q}(x_H^*E) =
H^{V+q}((G/H) \times F) = \underline{H}^{V+q}(F;M)(G/H)$.

\end{proof}


\section{Homotopical Considerations}
\label{sec:homotopical}

What follows is an equivariant version of some of the statements
about homotopical decompositions in \cite{hocolim}.  These are
needed for the proof of Lemma \ref{lemma:hocowe}, and may also be of
independent interest.  Because of the technical nature of this
material, the uninterested reader may skip ahead to the next
section.

Consider the category $G$-$\sSet$ of equivariant simplicial sets.
 The objects are simplicial sets endowed with a $G$-action and all of
the face and degeneracy maps respect the action.  The morphisms are
equivariant versions of the usual simplicial maps.  An $n$-simplex
of an equivariant simplicial set $X$ is an element $\sigma \in X_n$.
Alternatively, we can think of an equivariant $n$-simplex as an
equivariant simplicial map $\sigma \colon G/H \times \Delta^n \ra
X$.  Both points of view can be useful.

$G$-$sSets$ has a model category structure in which fibrations
andweak equivalences are defined in terms of the fixed sets, that is
$f$ is a fibration if for all subgroups $H$ the simplicial map $f^H$
is a fibration, and similarly for weak equivalences.  The
cofibrations are then the maps with the appropriate lifting
properties.

Let $D \colon I \ra G$-$\sSet$ be a diagram of equivariant
simplicial sets.  Suppose there is a map $\colim_I D \ra X$.  For
each simplex $\sigma \in X$ let $F(D)_\sigma$ denote the category
whose objects are pairs $[i, \alpha \in (D_i)_n]$ such that the map
$D_i \ra X$ sends $\alpha$ to $\sigma$.  A map in $F(D)_\sigma$ from
$[i, \alpha \in (D_i)_n]$ to $[j, \beta \in (D_j)_n]$ is a map $i
\ra j$ such that $D_i \ra D_j$ sends $\alpha$ to $\beta$.  Then
$F(D)_\sigma$ is called the fiber category of $D$ over $\sigma$.

\begin{prop}  Suppose that $D \colon I \ra G$-$\sSet$ and $X$ are as above, and assume that for every $n \geq 0$ and every $\sigma \in X_n$ the fiber category $F(D)_\sigma$ is contractible.  Then the map $\hocolim D \ra X$ is a weak equivalence of equivariant simplicial sets.
\end{prop}

\begin{proof}
The proof is nearly identical to that of Proposition 16.9 in
\cite{hocolim}.  The key facts are that for a bisimplicial set $B$,
the geometric realization satisfies $|B|^H=|B^H|$ and that weak
equivalences are determined by their fixed sets.
\end{proof}

Now, suppose that $D \colon I \ra \Top_G$ is a diagram of
$G$-spaces.  Suppose we have a map $p \colon \colim D \ra X$.  Then
for each $n \geq 0$, each subgroup $H \leq G$ and each $\sigma
\colon G/H \times \Delta^n \ra X$ define the fiber category
$F(D)_\sigma$ of $D$ over $\sigma$ to be the category with objects
pairs $[i, \alpha \colon G/H \times \Delta^n \ra D_i]$ such that $p
\circ \alpha = \sigma$.  A map from $[i, \alpha \colon G/H \times
\Delta^n \ra D_i]$ to $[j, \beta \colon G/H \times \Delta^n \ra
D_j]$ is a map $i \ra j$ making the obvious diagram commute.

\begin{prop}  In the above setting, suppose that for each $n\geq 0$, $H \leq G$, and $\sigma \colon G/H \times \Delta^n \ra X$ the category $F(D)_\sigma$ is contractible.  Then the composite $\hocolim D \ra \colim D \ra X$ is a weak equivalence.
\end{prop}
\begin{proof}
A map $\sigma \colon G/H \times \Delta^n \ra X$ is equivalent to a
map $\bar{\sigma} \colon \Delta^n \ra X^H$.  Thus we can reduce to
looking at the fixed sets.  But, this is exactly Theorem 16.2 in
\cite{hocolim}.  The condition that $F(D)_\sigma$ is contractible is
equivalent to the condition that $F(D)_{\bar{\sigma}}$ is
contractble.  Thus the composite is a weak equivalence on fixed
sets, and so is an equivariant weak equivalence.
\end{proof}

There is a related simplicial version of the above theorem.  Assume
that in addition there is a diagram $\tilde{D} \colon I \ra
G$-$\sSet$ and a natural isomorphism $\phi_i \colon |\tilde{D_i}|
\ra D_i$.  For each $\sigma \colon G/H \times \Delta^n \ra X$ define
the category $\tilde{F}(D)_\sigma$ to have objects pairs $[i, G/H
\times \Delta^n_s \ra \tilde{D}_i]$ such that the composite $|G/H
\times \Delta^n_s| \ra |\tilde{D}_i| \ra D_i \ra X$ is $\sigma$. The
morphisms are as expected.  Here, $\Delta^n_s \in \sSet$ is the
$n$-simplex.  The following is a refinement of the previous theorem.

\begin{prop}
\label{prop:simphoco} In the above setting, suppose that for each $n
\geq 0$, $H \leq G$, and $\sigma \colon G/H \times \Delta^n \ra X$
the category $\tilde{F}(D)_\sigma$ is contractible.  Then the
composite $\hocolim D \ra \colim D \ra X$ is a weak equivalence.
\end{prop}
\begin{proof}
Again, we can reduce to looking at fixed sets, this time invoking
Proposition 16.3 in \cite{hocolim}.
\end{proof}

For a $G$-fibration $f \colon E \ra X$, we can consider the diagram
$\Gamma_f \colon \Delta_G(X) \ra \Top$ that sends $\sigma \colon G/H
\times \Delta^n \ra X$ to the pullback $\sigma^*(E)$.  We then have
the following technical lemma used in the construction of the
spectral sequence.
\begin{lemma}
\label{lemma:hocowe} The map $\hocolim_{\Delta_G(X)} \Gamma_f \ra
\colim_{\Delta_G(X)} \Gamma_f \ra E$ is a weak equivalence.
\end{lemma}

\begin{proof}
Consider the diagram $D \colon \Delta_G(X) \ra G$-$\sSet$ sending
$([k], \alpha \colon G/H \times \Delta^k)$ to the simplicial set
obtained as the pull back $G/H \times \Delta^n_s \we S(G/H \times
\Delta^n) \ra S(X) \la S(E)$, where $S(-)$ is the singular functor.

There is a map of diagrams $|D| \ra \Gamma_f$ which is an objectwise
weak equivalence since $f$ is a fibration.  We are reduced to
showing that $\hocolim |D| \ra \colim |D| \ra E$ is a weak
equivalence.

For each $n \geq 0$, $H \leq G$, and $\sigma \colon G/H \times
\Delta^n \ra E$, the category $\tilde{F}(D)_\sigma$ is contractible.
This is due to the presence of an initial object associated to the
map $f \circ \sigma \colon G/H \times \Delta^n \ra X$.  By
Proposition \ref{prop:simphoco}, we are done.
\end{proof}

\section{Cohomology of Projective Bundles}
\label{sec:ProjBundles}


In this chapter, we specialize exclusively to the case where
$G=\Z/2$.

To any equivariant vector bundle $f\colon E \ra X$, there is an
associated equivariant projective bundle $\bP(f)\colon \bP(E)\ra X$
whose fibers are lines in the fibers of the original bundle.
Applying the spectral sequence of Theorem \ref{thm:localss} to this
new bundle yields the following result:
\begin{thm}
\label{thm:injection} If $X$ is equivariantly 1-connected and
$f\colon E \ra X$ is a vector bundle with fiber $\R^{n,m}$ over the
base point, then the spectral sequence of Theorem \ref{thm:localss}
for the bundle $\bP(f)\colon \bP(E)\ra X$ with constant
$M=\underline{\Z/2}$ coefficients ``collapses".
\end{thm}

Here, the spectral sequence ``collapses" in the sense that the only
nonzero differentials are those arising from the trivial fibration
$id \colon X \ra X$.  One can also view this collapsing by identifying the terms of the spectral sequence (at least after the $E_2$ page) as a certain tensor product.  Letting $E^{*,*}_*$ be the spectral sequence associated to the fibration $\bP(f)\colon \bP(E)\ra X$ above and $F^{*,*}_*$ be the one for $id\colon X \ra X$, then for $n\geq 2$ there is an isomorphism $E^{p,q}_n \cong (F^{*,*}_n \otimes H^{*,*}(\bP(R^{n,m})))^{p,q}$, where the tensor product is taken over $H^{*,*}(pt)$.

\begin{rmk} Despite the spectral sequence $E^{*,*}_*$ behaving like $F^{*,*}_* \otimes H^{*,*}(\bP(R^{n,m}))$, it need not be the case that $H^{*,*}(\bP(E))$ and $H^{*,*}(X)\otimes H^{*,*}(\bP(R^{n,m}))$ are isomorphic.
\end{rmk}

The projective spaces involved here have
actions on them induced by the action in the fibers.  Briefly, we
denote by $\R \bP^{n}_{tw}=\bP(\R^{n+1,\left\lfloor \frac{n+1}{2}
\right\rfloor})$, the equivariant space of lines in
$\R^{n+1,\left\lfloor \frac{n+1}{2} \right\rfloor}$.  For the other
projective spaces, we simply denote the space of lines in $\R^{n,m}$
by $\bP(\R^{n,m})$.  These projective spaces themselves are studied
in more detail in \cite{freeness}.  Some of the relevant results are
restated here for convenience.

\begin{lemma}
\label{lemma:easyRP} $\bP(\R^{p,q}) \cong \bP(\R^{p,p-q})$.

\end{lemma}

\begin{lemma}
\label{lemma:RPmodule} As a $H^{*,*}(pt)$-module,
$H^{*,*}(\bP(\R^{p,q}))$ is free with a single generator in
dimensions $(0,0)$, $(1,1)$, $(2,1)$, $(3,2)$, $(4,2), \dots,$
$(2q,q)$, $(2q+1,q), \dots,$ $(p-1,q)$.

\end{lemma}

\begin{thm}
\label{thm:rpinfty} $H^{*,*}(\R\bP^\infty_{tw})=H^{*,*}(pt)[a,
b]/(a^2=\rho a +\tau b)$, where $\deg(a)=(1,1)$ and $\deg(b)=(2,1)$.
\end{thm}

\begin{prop}
\label{prop:ringRP} $H^{*,*}(\bP(\R^{p,q}))$ is a truncated
polynomial algebra over $H^{*,*}(pt)$ on generators in dimensions
$(1,1)$, $(2,1)$, $(2q+1,q)$, $(2q+2,q), \dots,$ $(p-1,q)$, subject
to the relations determined by the restriction of
$H^{*,*}(\R\bP^\infty_{tw})$ to $H^{*,*}(\bP(\R^{p,q}))$.

\end{prop}

With these observations, we are ready to prove Theorem
\ref{thm:injection}.

\begin{proof}[Proof of Theorem \ref{thm:injection}]

By Lemma \ref{lemma:easyRP} we need only consider the case where
$n\geq m/2$.

First, consider the case where the vector bundle has fiber
$\R^{n,\left\lfloor \frac{n}{2} \right\rfloor}$ over the base point.

If $n$ is odd, consider the vector bundle $E \oplus \R^{1,0} \ra X$,
and if $n$ is even consider $E \oplus \R^{1,1} \ra X$. In either
case, denote this new bundle by $E \oplus L$.  Taking the associated
projective bundles gives a diagram

\[ \xymatrix{  \R\bP^{n-1}_{tw} \ar[r] \ar[d] & \R\bP^{n}_{tw} \ar[r] \ar[d] & pt \ar[d]\\
    \bP(E)  \ar[r] \ar[d] & \bP(E \oplus L) \ar  [r] _>>>>{\bP(f)}  \ar [d]_{\bP(f)}  & X \ar @/_1pc/ [l]<-0.5ex> _{s} \ar[d]^{id}\\
    X \ar[r]^{id} & X \ar[r]^{id} \ar @/_/ [u]<-0.5ex>_>>>>>s & X}
\]
Here and below $pt$ is the one point set with trivial $\Z/2$ action.
The map $s$ above is the canonical splitting that assigns to each
point $x$ in $X$ the line given by the trivial factor in $E \oplus
L$. It is important to note that this is indeed an equivariant
splitting in both the case of $\R^{1,0}$ and $\R^{1,1}$. This
diagram yields maps between the spectral sequences associated to
these three bundles over $X$.  Let us consider the $r=1$ spectral
sequence for $\bP(E \oplus L)$.  This is the sequence with
$E_2$-page given by $$E_2^{p,q} = H^{p,0}(X;
\underline{H}^{q,1}(\R\bP^{n}_{tw})) \Rightarrow H^{p+q,1}(\bP(E
\oplus L)).$$ This spectral sequence is generated as an algebra over
$H^{*,*}(pt)$ by the classes $a \in H^{0,0}(X; H^{1,1}(\R
\bP^{n}_{tw}))$ and $b \in H^{0,0}(X; H^{2,1}(\R \bP^{n}_{tw}))$. To
see that the spectral sequence collapses, we need only see that
these classes $a$ and $b$ have trivial differentials.

The splitting $s$ induces a map $s^*$ from the $r=1$ spectral
sequences associated to the bundle $f:\bP(E) \ra X$ to the one for
the trivial bundle $id \colon X \ra X$.  This map sends the class $a
\in H^{0,0}(X; H^{1,1}(\R\bP^{n}_{tw}))$ to $0 \in H^{0,0}(X;
H^{1,1}(pt))$ since $s^*\colon H^{*,*}(\R\bP^{n}_{tw}) \ra
H^{*,*}(pt)$ is the projection.  Thus $s^*da = d(s^* a) = 0$.
However, $s^*$ gives an isomorphism between coefficient systems
$\underline{H}^{0,1}(\R\bP^{n}_{tw}) \iso_{s^*}
\underline{H}^{0,1}(pt)$.  Observe:

\begin{tabular}{ll}
$\left(\underline{H}^{0,1}(\R\bP^{n}_{tw})\right)(\Z/2)$ & $=H^{0,1}(\Z/2 \times \R\bP^{n}_{tw})$ \\

& $=[\Z/2 \times \R\bP^{n}_{tw}, K(\Z/2(1),0)]_{\Z/2}$ \\

& $= [\R\bP^{n}_{tw},K(\Z/2,0)]_e$ \\

& $= H^0_{sing}(\R\bP^{n};\Z/2)$ \\

& $\iso_{s^*} H^0_{sing}(pt;\Z/2)$ \\

& $= H^{0,1}(pt)$ \\

& $= \left(\underline{H}^{0,1}(pt)\right)(\Z/2).$
\end{tabular}

Here, $K(A(q),p)$ is the representing space for $H^{p,q}( -
;\underline{A})$.

Also, $\underline{H}^{0,1}(\R\bP^{n}_{tw})(G/G)=
H^{0,1}(\R\bP^{n}_{tw}) \iso_{s^*} H^{0,1}(pt)$.  It now follows
that since $s^*da=0$, it must be that $da=0$.

Now, from the relation $a^2 = \rho a + \tau b$ we get that $0 =
d(a^2) = \rho da + \tau db$.  Hence $\tau db = 0$.  But,
$\tau\colon\underline{H}^{1,1}(\R\bP^{n}_{tw}) \ra
\underline{H}^{1,2}(\R\bP^{n}_{tw})$ is an isomorphism.  Thus
$db=0$.

Now, $\bP(i)^*(a) =a$ and $\bP(i)^*(b) =b$, where $\bP(i)^*: \bP(E
\oplus L) \ra \bP(E)$.  Thus $d(a)=0$ and $d(b)=0$ in the spectral
sequence for $\bP(E)$ as well. Therefore the spectral sequence
``collapses,'' in the sense that all differentials are zero, except
for the part of the spectral sequence corresponding to the trivial
fibration $id\colon X \ra X$.

For the other projective spaces, we can proceed inductively.  Fix
$m$ and induct on $n \geq m/2$.  The base case is exactly the
argument above.  For the inductive step, in going from
$\bP(\R^{n,m})$ to $\bP(\R^{n+1,m})$, a single new cohomology
generator $c_{n,m}$ appears in degree $(n,m)$, according to Lemma
\ref{lemma:RPmodule}.  Also, by Proposition \ref{prop:ringRP}, we
have $a c_{n-1,m} = \tau c_{n,m}$, where $c_{n-1,m}$ is the highest
dimensional cohomology generator in $H^{*,*}(\bP(\R^{n,m}))$.  Then
in the spectral sequence we have $d(a c_{n-1,m}) = \tau d(c_{n,m})$.
But, by induction, $d(a c_{n-1,m})=0$.  Since $\cdot \tau$ is still
an injection in the range we are working in, it must be that
$d(c_{n,m})=0$.  This gives the desired collapsing of the spectral
sequence.
\end{proof}

In fact, we can deduce even more about such a projective bundle.

\begin{cor} If $f\colon E \ra X$ is a vector bundle with $X$ equivariantly 1-connected with fiber $\R^{n,m}$ over the base point, then $\bP(F)^* \colon H^{*,*}(X) \ra H^{*,*}(\bP(E))$ is an injection.
\end{cor}
\begin{proof}
By the preceding theorem, there is a natural injection of the
spectral sequence for $id_X$ into the one for $\bP(f)$, thus an
injection on the filtrations. We get an injection on the $E_\infty$
terms, and thus, by the following lemma, an injection $H^{*,*}(X)
\ra H^{*,*}(\bP(E))$.
\end{proof}

\begin{lemma} Let $f \colon E^{p,q}_r \ra F^{p,q}_r$ be a map of first quadrant spectral sequences, converging to $A_{p+q}$ and $B_{p+q}$ respectively, which is an injection for every $p$, $q$, and $r$. Then $f$ induces an injection  $\tilde{f} \colon A_{p+q} \ra B_{p+q}$.
\end{lemma}

\begin{proof}
Fix $n$.  Then there is a filtration $0 \subseteq A_0 \subseteq
\cdots \subseteq A_n$ with $A_i/A_{i-1} \cong E^{n-i,i}_\infty$.
Similarly, there is a filtration $0 \subseteq B_0 \subseteq \cdots
\subseteq B_n$ with $B_i/B_{i-1} \cong F^{n-i,i}_\infty$.  Notice
that $A_0 = E^{n,0}_\infty$ and $B_0 = F^{n,0}_\infty$.  Thus
$\tilde{f}_0 \colon A_0 \ra B_0$ is injective.  Induction starts.

Suppose that $\tilde{f}_i \colon A_i \ra B_i$ is injective.  We also
know that $f_{i+1} \colon A_{i+1}/A_i \ra B_{i+1}/B_i$ is injective.
We have a map $\tilde{f}_{i+1} \colon A_{i+1} \ra B_{i+1}$ that
restricts to $\tilde{f}_i$ and we want to see that $\tilde{f}_{i+1}$
is injective.  Suppose $\tilde{f}_{i+1}(a) = 0$.  Then
$f_{i+1}([a])=0$.  But this map is injective, so $a \in A_i$.  Since
$\tilde{f}_{i+1}$ restricts to $\tilde{f}_i$ on $A_i$, we have that
$\tilde{f}_{i+1}(a) = \tilde{f}_i(a) = 0$.  As $\tilde{f}_i$ is
injective, $a=0$.  By induction, $\tilde{f}_n = \tilde{f}$ is
injective.

\end{proof}


\section{Equivariant Adams-Hilton Construction}
\label{sec:AH}

This sections provides a $G$-representation complex structure to the
space of Moore loops of a $G$-representation space $Y$ under certain
assumptions on the types of cells involved.

Let $(Y,*)$ be a based $G$-space.  Let $\Omega^M(Y,*) \subseteq
Map([0,\infty),Y) \times [0,\infty)$ denote the subspace of all
pairs $(\varphi,r)$ for which $\varphi(0)=*$ and $\varphi(t)=*$ for
$t \geq r$.  The space $\Omega^M(Y,*)$ is the space of Moore loops
of Y.  It inherits a $G$-action given by $g \cdot (\varphi,r) =
(g\cdot \varphi, r)$, where $(g\cdot\varphi)(t) = g\cdot
\varphi(t)$. (The action of $G$ on both $\R$ and $[0,\infty)$ are
assumed to be trivial, so this is the usual diagonal action of $G$
on a product restricted to the subspace of Moore loops.)

\begin{prop} $\Omega(Y,*)$ is a $G$-deformation retract of $\Omega^M(Y,*)$.

\end{prop}
\begin{proof}
The argument from nonequivariant topology adapts effortlessly to the
equivariant setting.  What follows is essentially the argument from
Proposition 5.1.1 of \cite{MS}.

First consider $\tilde{\Omega}(Y,*) \subseteq \Omega^M(Y,*)$, the
subspace of all $(\varphi,t)$ with $t \geq 1$.  A deformation
retraction, $H$, of $\Omega^M(Y,*)$ onto $\tilde{\Omega}(Y,*)$ is
given by the following formulae:

\hspace{0.5in}$H(s,(\varphi,r)) = (\varphi, r+s)$ when $r+s \leq 1$

\hspace{0.5in}$H(s,(\varphi,r)) = (\varphi, 1)$ when $r \leq 1$ and
$r+s \geq 1$

\hspace{0.5in}$H(s,(\varphi,r)) = (\varphi, r)$ when $r \geq 1$

Now a deformation retraction $K$ from $\tilde{\Omega}(Y,*)$ to
$\Omega(Y,*)$ is given by the formula

\hspace{0.5in}$K(s,(\varphi,r)) = (\varphi_s, (1-s)r +s)$,

\noindent where $\varphi_s(t) = \varphi(\frac{r}{(1-s)r+s}t)$.

Notice that $H$ and $K$ are both equivariant deformation
retractions.

\end{proof}

Given any based $G$-space $(X,*)$, one can form the free $G$-monoid
$M(X,*)$ just as in the nonequivariant setting.  As a space, $M(X,*)
= \coprod X^n / \sim $.  Here, $\sim$ is the equivalence relation
generated by all the relations of the form

$(x_1, \dots , x_{i-1}, *, x_{i+1}, \dots , x_n ) \sim (x_1, \dots ,
x_{i-1}, x_{i+1}, \dots , x_n )$.

The $G$-action on $M(X,*)$ is inherited from the diagonal action of
$G$ on each of the products $X^n$.  Note that since the basepoint
$*$ is fixed by $G$, this action factors through the relation
$\sim$.

This free $G$-monoid on $(X,*)$ enjoys the universal property that
any based $G$-map $f \colon X \ra M$, where $M$ is any topological
$G$-monoid with $f(*) = e$, can be extended uniquely to a $G$-monoid
map $\tilde{f} \colon M(X,*) \ra M$.

$\Omega^M(Y,*)$ is a topological $G$-monoid.  The loop concatenation
product respects the $G$ action in the sense that $g \cdot
((\varphi,r) * (\psi,s)) = ((g \cdot \varphi) * (g \cdot \psi),
r+s)$.  The point $(*,0)$, where $*$ denotes the contant loop at the
base point of $Y$, is the identity element.

 Let $(X,*)$ be a based $G$-space. The equivariant James map is the $G$-map $J \colon (X,*) \ra (\Omega \Sigma X,*)$ given by $J(x)(t) = [t,x] \in \Sigma X$.  Here, the $G$-action is given by $(g \cdot J(x))(t) = [t,g\cdot x]$.  Compose this $G$-map with the inclusion of $\Omega \Sigma X$ into $\Omega^M \Sigma X$ to obtain a $G$-map $J \colon (X,*) \ra (\Omega^M \Sigma X,*)$ that does not carry the base point to the identity.  Let $\hat{X} = X \coprod [0,1] / (1\sim*)$ and define an extension $\hat{J}$ of $J$ to $\hat{X}$ by $\hat{J}(s) = (*,s)$ for $s \in [0,1]$ where $*$ denotes the constant path at the basepoint.  Note that $\hat{X}$ and $X$ are based $G$-homotopy equivalent if $X$ is a $G$-CW complex.  By now considering $0$ to be the basepoint of $\hat{X}$, $\hat{J}$ is now a based $G$-map.  This now extends uniquely to a $G$-map $\bar{J} \colon M(\hat{X},0) \ra \Omega^M \Sigma X$.  This is the map in James' theorem.

James' Theorem states that if $X$ is a connected CW complex, the map
$\bar{J} \colon M(\hat{X},0) \ra \Omega^M \Sigma X$ is a homotopy
equivalence.  See \cite{CM} for a proof of James' theorem.  This can
be easily extended to the equivariant setting in the case that $X$
has connected fixed sets.

 \begin{thm}[Equivariant James Theorem] If $X$ is a connected $G$-CW complex with $X^H$ connected for all $H \leq G$, the $G$-map $\bar{J} \colon M(\hat{X},0) \ra \Omega^M \Sigma X$ is a $G$-homotopy equivalence.

 \end{thm}

\begin{proof}

Observe that $M(\hat{X},0)^H = M(\hat{X}^H,0)$ and $(\Omega^M \Sigma
X)^H = \Omega^M \Sigma (X^H)$.  Now, $\bar{J}^H \colon
M(\hat{X},0)^H \ra (\Omega^M \Sigma X)^H$ is a homotopy equivalence
by James' theorem since $X^H$ is a connected CW complex by
assumption.  Thus $\bar{J}$ is a $G$-homotopy equivalence.

\end{proof}

The space $J(X) = M(\hat{X},0)$ is called the James construction.
$J(X)$ is a free, associative, unital $G$-monoid.  If the basepoint
$*$ of $X$ is a vertex, then $J(X)$ has a natural $G$-CW complex
structure coming from the decomposition of $X^n$ as a product $G$-CW
complex.  Thus $J(X)$ has the following properties:

\begin{enumerate}
\item every element $v \in J(X)$ has a unique expression $v=*$ or $v=x_1x_2 \cdots x_n$, $x_i \in X \setminus *$ for $1 \leq i \leq n$.
\item $x_1 \cdots x_n$ is contained in a unique cell of $J(X)$, the cell $C_1 \times \cdots \times C_n$ where $x_i \in \text{Int}(C_i)$, $1 \leq i \leq n$, so that no indecomposable cell contains decomposable points, and
\item non-eqivariantly, the cell complex has the form of a tensor algebra $T(C_\#(X))$, where the sub complex $C_\#(X)$ is exactly the indecomposables, and the generating cells in dimension $i$ are in bijective correspondence with the cells in dimension $i+1$ of $\Sigma X$.

\end{enumerate}

\begin{rmk} The James construction given above uses the trivial one
dimensional representation for loops and suspensions.  This differs
from the construction in \cite{Rybicki} where the nontrivial one
dimensional $\Z/2$ representation is used.

\end{rmk}

Nonequivariantly, we have the Adams-Hilton construction as follows.
Let $Y$ be a CW complex with a single vertex $*$ and no 1-cells.
Then there is a model for $\Omega^M(Y)$ which is a free associative
monoid, with $*$ the only vertex, the generating cells in dimension
$i$ are in 1-1 correspondence with the $(i+1)$-dimensional cells of
$Y$, and it satisfies (2) above. This will generalize to the
following equivariant version.

\begin{thm}[Equivariant Adams-Hilton]  Let $Y$ be a $\text{Rep}(G)$-complex with a single vertex $*$, no 1-cells, and the only cells in higher dimensions are $V\oplus1$-cells where $V$ is a real representation of $G$ with all fixed sets of $S^V$ connected.  Then there is a model for $\Omega^M(Y)$ which is a free associative monoid, with $*$ the only vertex, the generating cells in dimension $V$ are in 1-1 correspondence with the $(V\oplus1)$-dimensional cells of $Y$, and it satisfies (2) above.
\end{thm}
For the case $G= \Z/2$ the theorem becomes the following.

\begin{thm}[$\Z/2$-Equivariant Adams-Hilton]  Let $Y$ be a $\text{Rep}(\Z/2)$-complex with a single vertex $*$ and no $(n,n)$-cells or $(n,n-1)$-cells for $n \geq1$.  Then there is a model for $\Omega^M(Y)$ which is a free associative monoid, with $*$ the only vertex, the generating cells in dimension $(p,q)$ are in 1-1 correspondence with the $(p+1,q)$-dimensional cells of $Y$, and it satisfies (2) above.
\end{thm}

\begin{proof}
With these restrictions on the types of cells in our
$\text{Rep}(G)$-complex, the proof of the Adams-Hilton theorem in
\cite{CM} adapts to the equivariant case.  For example, in the base
case of the inductive argument, one has that the 2-skeleton
$Y^{(2)}=\bigvee_{V_\alpha}S^{V_\alpha \oplus 1}=\Sigma^1(
\bigvee_{V_\alpha}S^{V_\alpha})$.  Since each $S^{V_\alpha}$ has
connected fixed sets, the equivariant James construction applies and
the result is immediate.

For the inductive step, the prolongation construction is equivariant and
so by \cite{WanerFibrations} the quasifibering arguments extend to the equivariant setting.  This allows the remainder of the argument to adapt to the equivariant setting.
\end{proof}


One application of this model is the computation of $H^{*,*}(\Omega
S^{p,q}; \underline{\Z/2})$ when $S^{p,q}$ has a connected fixed set
and $p \geq 2$.

\begin{prop}
If $S^{p,q}$ is equivariantly 1-connected, then $H^{*,*}(\Omega
S^{p,q}; \underline{\Z/2})$ is an exterior algebra over $H^{*,*}(pt;
\underline{\Z/2})$ on generators $a_1, a_2, \dots$, where $a_{i} \in
H^{(p-1)\cdot 2^{i-1},q\cdot 2^{i-1}}(\Omega
S^{p,q};\underline{\Z/2})$.
\end{prop}

\begin{proof}[Sketch of proof.]  For each value of $p$ and $q$, the argument is similar, so let's focus on the case $p=4$ and $q=2$ to compute $H^{*,*}(\Omega S^{4,2})$.

Now, since the fixed set of $S^{4,2}$ is connected, by the
Adams-Hilton construction we have an upperbound for the cohomology
of the loop space given in Figure \ref{fig:loops42}.

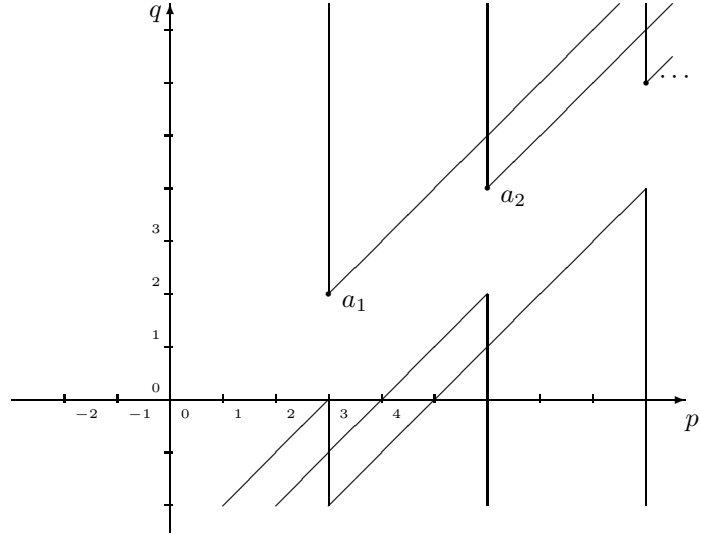
\begin{figure}[htpb]
\centering
\begin{picture}(330,230)(-150,-110)
\put(-110,-50){\vector(1,0){255}}
\put(-50,-100){\vector(0,1){200}}

\put(10, -10){\line(0,1){110}} \put(10, -10){\line(1,1){110}}
\put(10, -50){\line(0,-1){40}} \put(10, -50){\line(-1,-1){40}}
\put(10, -10){\circle*{2}} \put(15, -15){$a_1$}

\put(70, 30){\line(0,1){70}} \put(70, 30){\line(1,1){70}} \put(70,
-10){\line(0,-1){80}} \put(70, -10){\line(-1,-1){80}} \put(70,
30){\circle*{2}} \put(75, 25){$a_2$}

\put(130, 70){\line(0,1){30}} \put(130, 70){\line(1,1){10}}
\put(130, 30){\line(0,-1){120}} \put(130, 30){\line(-1,-1){120}}
\put(130, 70){\circle*{2}} \put(135,70){$\cdots$}

\multiput(-90,-51)(20,0){12}{\line(0,1){3}}
\multiput(-52,-90)(0,20){10}{\line(1,0){3}}
\put(-46,-57){$\scriptscriptstyle{0}$}
\put(-26,-57){$\scriptscriptstyle{1}$}
\put(-6,-57){$\scriptscriptstyle{2}$}
\put(14,-57){$\scriptscriptstyle{3}$}
\put(34,-57){$\scriptscriptstyle{4}$}
\put(-66,-57){$\scriptscriptstyle{-1}$}
\put(-86,-57){$\scriptscriptstyle{-2}$}
\put(-57,-47){$\scriptscriptstyle{0}$}
\put(-57,-27){$\scriptscriptstyle{1}$}
\put(-57,-7){$\scriptscriptstyle{2}$}
\put(-57,13){$\scriptscriptstyle{3}$}

\put(-58,95){${q}$} \put(145,-60){${p}$}
\end{picture}

\caption{The $E_1$ page of the cellular spectral sequence for
$\Omega S^{4,2}$.} \label{fig:loops42}
\end{figure}

In the spectral sequence of the filtration, it is clear that all
differentials must be zero, and so Figure \ref{fig:loops42} reveals
the structure of $H^{*,*}(\Omega S^{4,2})$ as a free
$H^{*,*}(pt)$-module.  Denote the generators of $H^{3\cdot
2^{i-1},2\cdot 2^{i-1}}(\Omega S^{p,q};\underline{\Z/2})$ by $a_i$.

Consider the path-loop fibration $\Omega S^{4,2} \ra PS^{4,2} \ra
S^{4,2}$.  The base is 1-connected, so we can apply the spectral
sequence of Theorem \ref{thm:localss}, which will converge to the
cohomology of a point since the total space $PS^{4,2} \simeq pt$.
Consider first the $r=2$ portion of the spectral sequence.

To fill in the entries in the spectral sequence, the Mackey functors
$\underline{H}^{q,2}(\Omega S^{4,2})$ need to be computed for
various values of $q$.  These can be obtained from the module
structure above.  The calculations yield that
$\underline{H}^{0,2}(\Omega S^{4,2}) = \underline{H}^{3,2}(\Omega
S^{4,2}) = \underline{\Z/2}$, $\underline{H}^{1,2}(\Omega S^{4,2})
=\underline{H}^{2,2}(\Omega S^{4,2}) = \langle\Z/2\rangle$, and
$\underline{H}^{4,2}(\Omega S^{4,2}) =\underline{H}^{5,2}(\Omega
S^{4,2}) =0$.  The Mackey functor $\underline{H}^{6,2}(\Omega
S^{4,2})$ is dual to $\underline{\Z/2}$, though this information
will not be needed.

Given the above Mackey functors, we have that the $q=0$ and $q=3$
rows are $H^{*,0}(S^{4,2};\Z/2)$, the $q=1$ and $q=2$ rows are
$H^*_{sing}(S^2;\Z/2)$, and the $q=4$ and $q=5$ rows are entirely
zeroes.  Thus the spectral sequence is as shown in Figure
\ref{fig:loops42r2ss}.

\begin{figure}[htpb]
\centering
\begin{picture}(340,180)(-30,-10)
\multiput(0,0)(50,0){6}{\line(0,1){150}}
\multiput(0,0)(0,20){8}{\line(1,0){280}}
\put(15,5){$\Z/2$} \put(15,25){$\Z/2$} \put(15,45){$\Z/2$}
\put(15,65){$\Z/2$} \put(20,85){0} \put(20,105){0} \put(15,125){??}

\put(75,5){0} \put(75,25){0} \put(75,45){0} \put(75,65){0}
\put(75,85){0} \put(75,105){0} \put(75,125){??}

\put(125,5){0} \put(120,25){$\Z/2$} \put(120,45){$\Z/2$}
\put(125,65){0} \put(125,85){0} \put(125,105){0} \put(125,125){??}

\put(175,5){0} \put(175,25){0} \put(175,45){0} \put(175,65){0}
\put(175,85){0} \put(175,105){0} \put(175,125){??}

\put(220,5){$\Z/2$} \put(225,25){0} \put(225,45){0}
\put(220,65){$\Z/2$} \put(225,85){0} \put(225,105){0}
\put(225,125){??}

\put(265,5){0} \put(265,25){0} \put(265,45){0} \put(265,65){0}
\put(265,85){0} \put(265,105){0} \put(265,125){??}

\put(-2,-2){$\bullet$} \thicklines \put(0,0){\vector(0,1){155}}
\put(0,0){\vector(1,0){285}}

\put(20,-10){0} \put(75,-10){1} \put(125,-10){2} \put(175,-10){3}
\put(225,-10){4}

\put(-10,5){0} \put(-10,25){1} \put(-10,45){2} \put(-10,65){3}
\put(-10,85){4} \put(-10,105){5} \put(-10,125){6}

\put(280,-10){$p$} \put(-10,150){$q$} \thinlines
\end{picture}

\caption{The $r=2$ spectral sequence for $\Omega S^{4,2} \ra
PS^{4,2} \ra S^{4,2}$.} \label{fig:loops42r2ss}
\end{figure}

Since the total space of the fibration is contractible, the spectral
sequence converges to $H^{p+q,2}(pt)$.  Since $H^{4,2}(pt) =0$,
there must be a nontrivial differential $d_2 \colon E^{0,3} \ra
E^{2,2}$ sending the generator $a_1 \in
H^{0,0}(S^{4,2};\underline{H}^{3,2}(\Omega S^{4,2}))$ to the
generator $z \in H^{2,0}(S^{4,2};\underline{H}^{2,2}(\Omega
S^{4,2}))$.

Now, the products $a_1^2$ and $a\cdot z$ live in the $r=4$ spectral
sequence and so to determine the differentials on $a_1^2$, we need
the picture of that spectral sequence.  This is shown in Figure
\ref{fig:loops42r4ss}.

\begin{figure}[htbp]
\centering
\begin{picture}(340,180)(-30,-10)
\multiput(0,0)(50,0){6}{\line(0,1){150}}
\multiput(0,0)(0,20){8}{\line(1,0){280}}
\put(15,5){$\Z/2$} \put(15,25){$\Z/2$} \put(15,45){$\Z/2$}
\put(10,65){$(\Z/2)^2$} \put(10,85){$(\Z/2)^2$} \put(15,105){$\Z/2$}
\put(15,125){$\Z/2$}

\put(75,5){0} \put(75,25){0} \put(75,45){0} \put(75,65){0}
\put(75,85){0} \put(75,105){0} \put(75,125){0}

\put(125,5){0} \put(120,25){$\Z/2$} \put(120,45){$\Z/2$}
\put(120,65){$\Z/2$} \put(115,85){$(\Z/2)^2$} \put(120,105){$\Z/2$}
\put(125,125){0}

\put(175,5){0} \put(175,25){0} \put(175,45){0} \put(175,65){0}
\put(175,85){0} \put(175,105){0} \put(175,125){0}

\put(220,5){$\Z/2$} \put(225,25){0} \put(225,45){0}
\put(220,65){$\Z/2$} \put(225,85){0} \put(225,105){0}
\put(220,125){$\Z/2$}

\put(265,5){0} \put(265,25){0} \put(265,45){0} \put(265,65){0}
\put(265,85){0} \put(265,105){0} \put(265,125){0}

\put(-2,-2){$\bullet$} \thicklines \put(0,0){\vector(0,1){155}}
\put(0,0){\vector(1,0){285}}

\put(20,-10){0} \put(75,-10){1} \put(125,-10){2} \put(175,-10){3}
\put(225,-10){4}

\put(-10,5){0} \put(-10,25){1} \put(-10,45){2} \put(-10,65){3}
\put(-10,85){4} \put(-10,105){5} \put(-10,125){6}

\put(280,-10){$p$} \put(-10,150){$q$} \thinlines
\end{picture}

\caption{The $r=4$ spectral sequence for $\Omega S^{4,2} \ra
PS^{4,2} \ra S^{4,2}$.} \label{fig:loops42r4ss}
\end{figure}

Since $H^{7,4}(pt)=0$, there must be a nontrivial differential $d_2
\colon E^{0,6} \ra E^{2,5}$ sending the generator $a_2$
isomorphically to $a \cdot z$.  Since $d_2(a_1^2) = 0$, it must be
that $a_1^2 = 0$.    An inductive argument will show that the ring
structure is indeed that of an exterior algebra with the specified
generators.
\end{proof}

\newpage
\bibliographystyle{alpha}
\bibliography{references}

\begin{thebibliography}{LMM81}

\bibitem[Bre67]{Bredon}
Glen~E. Bredon.
\newblock Equivariant cohomology theories.
\newblock {\em Bull. Amer. Math. Soc.}, 73:266--268, 1967.

\bibitem[Car99]{Caruso}
Jeffrey~L. Caruso.
\newblock Operations in equivariant {${\mathbb Z}/p$}-cohomology.
\newblock {\em Math. Proc. Cambridge Philos. Soc.}, 126(3):521--541, 1999.

\bibitem[CM95]{CM}
Gunnar Carlsson and R.~James Milgram.
\newblock Stable homotopy and iterated loop spaces.
\newblock In {\em Handbook of algebraic topology}, pages 505--583.
  North-Holland, Amsterdam, 1995.

\bibitem[Dug05]{DuggerKR}
Daniel Dugger.
\newblock An {A}tiyah-{H}irzebruch spectral sequence for {$KR$}-theory.
\newblock {\em $K$-Theory}, 35(3-4):213--256 (2006), 2005.

\bibitem[Dug08]{hocolim}
Daniel Dugger.
\newblock A primer on homotopy colimts.
\newblock {\em Preprint}, 2008.

\bibitem[FL04]{FL}
Kevin~K. Ferland and L.~Gaunce Lewis, Jr.
\newblock The {$R{\rm O}(G)$}-graded equivariant ordinary homology of
  {$G$}-cell complexes with even-dimensional cells for {$G={\mathbb Z}/p$}.
\newblock {\em Mem. Amer. Math. Soc.}, 167(794):viii+129, 2004.

\bibitem[Kro]{freeness}
W.~C. Kronholm.
\newblock A freeness theorem for {$R{\rm O}(\mathbb{Z}/2)$}-graded cohomology.
\newblock {\em Preprint}.

\bibitem[LMM81]{LMM}
G.~Lewis, J.~P. May, and J.~McClure.
\newblock Ordinary {$RO(G)$}-graded cohomology.
\newblock {\em Bull. Amer. Math. Soc. (N.S.)}, 4(2):208--212, 1981.

\bibitem[May96]{Alaska}
J.~P. May.
\newblock {\em Equivariant homotopy and cohomology theory}, volume~91 of {\em
  CBMS Regional Conference Series in Mathematics}.
\newblock Published for the Conference Board of the Mathematical Sciences,
  Washington, DC, 1996.
\newblock With contributions by M. Cole, G. Comeza\~na, S. Costenoble, A. D.
  Elmendorf, J. P. C. Greenlees, L. G. Lewis, Jr., R. J. Piacenza, G.
  Triantafillou, and S. Waner.

\bibitem[MS93]{MS}
I.~Moerdijk and J.-A. Svensson.
\newblock The equivariant {S}erre spectral sequence.
\newblock {\em Proc. Amer. Math. Soc.}, 118(1):263--278, 1993.

\bibitem[Ryb91]{Rybicki}
S{\l}awomir Rybicki.
\newblock {$Z\sb 2$}-equivariant {J}ames construction.
\newblock {\em Bull. Polish Acad. Sci. Math.}, 39(1-2):83--90, 1991.

\bibitem[Wan80]{WanerFibrations}
Stefan Waner.
\newblock Equivariant fibrations and transfer.
\newblock {\em Trans. Amer. Math. Soc.}, 258(2):369--384, 1980.

\end{thebibliography}

\end{document}